\documentclass[11pt]{article}

\usepackage[T1]{fontenc}
\usepackage[utf8]{inputenc}
\usepackage{lmodern}
\usepackage[margin=1in]{geometry}
\usepackage{amsmath,amssymb,amsthm,mathtools}
\usepackage{microtype}

\usepackage[hidelinks]{hyperref}

\newtheorem{theorem}{Theorem}
\newtheorem{lemma}[theorem]{Lemma}

\newcommand{\ex}{\operatorname{ex}}
\newcommand{\BC}{\mathrm{BC}}

\title{A note on the extremal number of Berge-\texorpdfstring{$C_4$}{C4}}
\author{
Nika Salia\thanks{
Theoretical Computer Science Department, Faculty of Mathematics and Computer Science, 
Jagiellonian University, Kraków, Poland. 
E-mail: \texttt{nika.salia@uj.edu.pl}.
}
\and
Casey Tompkins\thanks{
E-mail: \texttt{casey.tompkins@renyi.hu}}
}
\date{}

\begin{document}
\maketitle

\begin{abstract}
We improve the known upper bound for the extremal number of Berge-$C_4$-free $3$-uniform hypergraphs. More precisely, we prove that every $n$-vertex $3$-uniform hypergraph with no Berge cycle of length four has at most
\[
        \frac{n^{3/2}}{2+\sqrt2}+O(n)
\]
hyperedges. This improves the previous best-known leading constant $1/\sqrt{10}$ to $1/(2+\sqrt2)$.
\end{abstract}

\section{Introduction}

The extremal number of $C_4$ in graphs is a classical problem in graph theory.  The theorem of K\H{o}v\'ari, S\'os and Tur\'an~\cite{kHovari1954problem} 
and the finite-geometric constructions of Erd\H{o}s, R\'enyi and
S\'os~\cite{erdHos1966problem} yields the asymptotically sharp estimate
\[
        \ex(n,C_4)=\left(\frac12+o(1)\right)n^{3/2}.
\]
In this note, we investigate the corresponding problem for Berge $4$-cycles in $3$-uniform hypergraphs.

A Berge cycle of length $\ell$ consists of distinct vertices
$v_1,\ldots,v_\ell$ and $\ell$ distinct hyperedges $e_1,\ldots,e_\ell$ such that
\[
        \{v_i,v_{i+1}\}\subset e_i
\]
for every $i$, where indices are taken modulo $\ell$.  This notion originates in Berge~\cite{berge1985graphs}.  Let $\ex_r(n,\BC_\ell)$ denote the maximum number of edges in an $r$-uniform $n$-vertex hypergraph containing no Berge
cycle of length $\ell$. 

Gy\H{o}ri~\cite{gyHori2006triangle} proved that a Berge-$C_3$-free $3$-uniform hypergraph on
$n$ vertices has at most $n^2/8$ edges for all sufficiently large $n$, and this
bound is asymptotically sharp.  Write $\ex_r(n,\BC_\ell)$ for the maximum number of edges in an $n$-vertex, $r$-uniform hypergraph with no Berge cycle of length $\ell$. Gy\H{o}ri and Lemons proved more generally that,
\[
        \ex_r(n,\BC_\ell)=O_{r,\ell}\!\left(n^{1+1/\lfloor \ell/2\rfloor}\right)
\]
for every fixed $r\ge3$ and $\ell\ge3$~\cite{gyHori2012hypergraphs}.
We will focus on the case of $\ex_3(n,BC_4)$. F\"uredi and
\"Ozkahya obtained the bound
\[
        \ex_3(n,\BC_4)\le (1+o(1))\frac{2n^{3/2}}{3}
\]
as part of their work on Berge cycles of prescribed length~\cite{furedi20173}.  Ergemlidze, Gy\H{o}ri, Methuku, Salia and Tompkins~\cite{ergemlidze20203} later improved this to
\[
        \ex_3(n,\BC_4)\le (1+o(1))\frac{n^{3/2}}{\sqrt{10}}.
\]
Our main result is the following improvement of the leading coefficient.

\begin{theorem}\label{thm:main}
For every $n$,
\[
        \ex_3(n,\BC_4)\le \frac{n^{3/2}}{2+\sqrt2}+O(n).
\]
\end{theorem}

The best-known lower bound has a leading constant $1/(3\sqrt3)$ and comes from
a construction of Bollob\'as and Gy\H{o}ri~\cite{bollobas2008pentagons}.
Start with an extremal $C_4$-free bipartite graph with color classes of size
$n/3$ and $(1+o(1))n^{3/2}/(3\sqrt3)$ edges.  Fix one color class, and for each vertex $v$ contained in that class, add a new vertex $v'$.  Then, replace each edge $vw$ incident with $v$ by the triple $\{v,v',w\}$.  The resulting $3$-uniform hypergraph has
$n$ vertices, has $(1+o(1))n^{3/2}/(3\sqrt3)$ edges, and contains no
Berge-$C_4$.

The proof of Theorem~\ref{thm:main} is a refinement the path-counting method of Ergemlidze, Gy\H{o}ri,
Methuku, Salia and Tompkins~\cite{ergemlidze20203}.

\section{Proof of Theorem~\ref{thm:main}}
For a hypergraph $H$, its $2$-shadow $\partial H$ is the graph on $V(H)$ in which $xy$ is an edge if some hyperedge of $H$ contains $\{x,y\}$.

We say that two hyperedges $e,f\in E(H)$ are equivalent if there is a sequence $e=e_1,e_2,\ldots,e_k=f$ of hyperedges of $H$ such that $|e_i\cap e_{i+1}|=2$ for every $i$. A \emph{block} of $H$ is the subhypergraph formed by one equivalence class of hyperedges, with vertex set equal to the union of those hyperedges. Thus, the blocks form a partition of $E(H)$. Every edge of $\partial H$ belongs to the $2$-shadow of a unique block.

The following block classification was used~\cite{ergemlidze20203}.
Let $H$ be a Berge-$C_4$-free $3$-uniform hypergraph on $n$ vertices, let $m=|E(H)|$, and put $G=\partial H$.  Every block $B$ of $H$ is either isomorphic to
$K_4^{(3)-}$, the $3$-uniform hypergraph on four vertices with three edges, or has the
following star-like form: there is an edge $e\in E(B)$ such that every other
edge of $B$ shares two vertices with $e$, and any two non-central edges
intersect only inside $e$.
Consequently,
\begin{equation}\label{Eq:block}
        |V(B)|\ge |E(B)|,
        \qquad
        |E(\partial B)|\ge 2|E(B)|
\end{equation}
for every block $B$.

For $v\in V(H)$, write $d_s(v)$ for the degree of $v$ in $G$, and let $d_b(v)$ be the number of blocks containing $v$ in $H$. By \eqref{Eq:block}, we have
\begin{equation}\label{eq:sums_low}       \sum_v d_s(v)=2|E(G)|=2\sum_B |E(\partial B)|
        \ge 2\sum_B 2|E(B)|=4m,   
         \qquad
        \sum_v d_b(v)=\sum_B |V(B)|\ge m.
\end{equation}
A hyperedge contributes at most three shadow edges
and is part of a unique block.  Hence
\[
    \sum_v d_s(v)\le 6m,
        \qquad
        \sum_v d_b(v)\le 3m.
\] 
For a block $B$ containing $v$, let $N_B(v)=\{x\in V(H): vx\in E(\partial B)\}$ and $a_B(v)=|N_B(v)|$. The sets $N_B(v)$ over all blocks containing $v$ partition $N_G(v)$, and $a_B(v)\ge2$ for every such $B$. For each vertex $v$, enumerate the blocks containing $v$ as $B_1,\ldots,B_{d_b(v)}$, and write $a_i(v)=a_{B_i}(v)$. Define
\[
        P=\sum_v\binom{d_s(v)}2,
        \qquad
        Q=\sum_v\sum_{1\le i<j\le d_b(v)}a_i(v)a_j(v).
\]
Thus, \(P\) is the total number of \(3\)-vertex paths in \(G\).
Among these paths, \(Q\) counts precisely those whose two edges lie in the shadows of distinct blocks. We call such paths \emph{cross-block paths}.

We say that a $4$-cycle $x_1x_2x_3x_4x_1$ in $G$ is \emph{rare} if the induced hypergraph $H[\{x_1,x_2,x_3,x_4\}]$ contains no two distinct hyperedges sharing diagonal pair $x_1x_3$ or $x_2x_4$. We say that a $3$-vertex path $xyz$ in $G$ is \emph{good} if $\{x,y,z\}\notin E(H)$ and there is no vertex $w$ such that $wxyzw$ is a rare $4$-cycle.

The next three lemmas are Claim~3, Claim~4, and the subsequent non-good path count in \cite{ergemlidze20203}. We include short proofs for completeness.

\begin{lemma}\label{lem:two-paths}
For any two distinct vertices $a,b\in V(H)$, the number of vertices
$v\in V(H)\setminus\{a,b\}$ for which $avb$ is a path in $G$ and
$\{a,b,v\}\notin E(H)$ is at most two.  Consequently, there are at most two
good paths with any prescribed pair of endpoints.
\end{lemma}

\begin{proof}
Suppose $av_i b$ for $i=1,2,3$ are three such paths. Choose hyperedges $e_i\supset\{a,v_i\}$ and $f_i\supset\{b,v_i\}$. Since $\{a,v_i,b\}\notin E(H)$, we have $e_i\ne f_i$. Also $e_i\ne f_j$ for $i\ne j$, since one triple cannot contain $a,v_i,b,v_j$. Relabel so that $e_1\ne e_2,e_3$ and $f_1\ne f_2$. Then $e_1,f_1,f_2,e_2$ are distinct and form the Berge cycle $ae_1v_1f_1bf_2v_2e_2a$, a contradiction.
\end{proof}

\begin{lemma}\label{lem:rare}
There are at most $6m$ rare $4$-cycles in $G$.
\end{lemma}

\begin{proof}
Every $4$-cycle in $G$ can be represented by some hyperedge contained in its four vertices.
Otherwise, its four shadow edges would be supported by four distinct
hyperedges, giving a Berge-$C_4$.  Fix such a representative
$e=\{a,b,c\}$.  If a rare cycle represented by $e$ has diagonal $ab$ and fourth
vertex $z$, then $azb$ is a path in $G$ not contained in a hyperedge.  By
Lemma~\ref{lem:two-paths}, there are at most two choices for $z$.  Applying the
same argument to the three pairs in $e$, the edge $e$ represents at most six
rare cycles. Thus, there are at most $6m$ rare $4$-cycles in $G$.
\end{proof}

\begin{lemma}\label{lem:nongood}
The number of non-good $3$-vertex paths in $G$ is at most $21m$.
\end{lemma}

\begin{proof}
A non-good path is either contained in a hyperedge of $H$, which are at most $3m$, or lies in a rare $4$-cycle. By Lemma \ref{lem:rare}, there are at most $6m$ rare $4$-cycles. Each has a representative hyperedge on three of its vertices, so each contributes at most three further paths. The total is at most $3m+3\cdot6m=21m$.
\end{proof}

The following lemma is the part of Claim~5 in \cite{ergemlidze20203} that we need, stated in terms of cross-block paths.

\begin{lemma}\label{lem:cross-one}
If $\{x,y\}$ is the endpoint pair of a cross-block path $xvy$, then there is at most one good path in $G$ with endpoints $x,y$.
\end{lemma}

\begin{proof}
Assume that, besides the cross-block path $xvy$, there is a good path $xuy$,
where $u\ne v$.  Choose supporting hyperedges
\[
e_x\supset\{x,v\},\quad e_y\supset\{v,y\},\quad
h_x\supset\{x,u\},\quad h_y\supset\{u,y\},
\]
with $e_x$ and $e_y$ in distinct blocks.  Since $xuy$ is good, $h_x\ne h_y$.
If $e_x,e_y,h_y,h_x$ were distinct, they would form a Berge-$C_4$  $xvyu$.  Hence, some coincidence must occur. As the hypergraph is $3$-uniform, we have either $h_x=e_x$ or $h_y=e_y$.


Assume \(h_x=e_x\), so \(e_x=\{x,u,v\}\). The case \(h_y=e_y\) is symmetric.
We show the cycle \(vxuyv\) is rare. 
The diagonals of \(vxuyv\) are \(vu\) and \(xy\). For the diagonal \(xy\), we have
\(\{x,u,y\}\notin E(H)\) since \(xuy\) is good. Also
\(\{x,v,y\}\notin E(H)\), since such a hyperedge would share \(xv\) with
\(e_x\) and \(vy\) with \(e_y\), forcing \(e_x\) and \(e_y\) into the same block.
For the diagonal \(vu\), the only possible second hyperedge besides
\(e_x=\{x,u,v\}\) is \(\{y,u,v\}\). This hyperedge is also impossible, since it
would share \(uv\) with \(e_x\) and \(vy\) with \(e_y\), again forcing \(e_x\)
and \(e_y\) into the same block. Thus \(vxuyv\) is rare, contradicting that \(xuy\) is good.

\end{proof}

Let $\mathcal D$ be the set of unordered pairs $\{x,y\}$ that occur as endpoints of at least one cross-block path, and let $D=|\mathcal D|$.  

\begin{lemma}\label{lem:DQ}
We have $D\ge Q-12m$.
\end{lemma}

\begin{proof}
For an unordered pair $\{x,y\}$, let $k(\{x,y\})$ be the number of middle vertices $v$ for which $xvy$ is a cross-block path. Then $Q=\sum_{\{x,y\}}k(\{x,y\})$ and $D=|\{\{x,y\}:k(\{x,y\})>0\}|$. Fix a pair with $k(\{x,y\})>0$ and choose one cross-block middle vertex $v_0$. For every other such middle vertex $w$, the cycle $xv_0ywx$ is rare. Indeed, if the shared diagonal is \(xy\), then \(\{x,y,v_0\}\) joins the two blocks at
\(v_0\). 
If it is \(v_0w\), then \(\{v_0,w,x\}\) and \(\{v_0,w,y\}\) do the
same. In both cases, this contradicts that \(xv_0y\) is cross-block.
Hence, the \(k(\{x,y\})-1\) additional middle vertices yield distinct rare cycles. Since each rare cycle has only two pairs of opposite vertices, it is charged to at most two endpoint pairs. By Lemma~\ref{lem:rare}, $Q-D\le 2\cdot 6m,$
and the claim follows.
\end{proof}

We now upper-bound $P+Q$. By Lemmas \ref{lem:two-paths} and \ref{lem:cross-one}, each pair of vertices is the endpoint pair of at most two good paths, and each pair in $\mathcal D$ is the endpoint pair of at most one. Hence, the number of good paths is at most $2\binom n2-D$. With Lemmas \ref{lem:nongood} and \ref{lem:DQ},
\[
        P\le 2\binom n2-D+21m\le 2\binom n2-Q+33m,
\]
so
\begin{equation}\label{eq:upperPQ}
        P+Q\le n^2+33m.
\end{equation}

It remains to lower-bound $P+Q$. Fix a vertex $v$
and recall the blocks containing $v$ are $B_1,\ldots,B_{d_b(v)}$, and we write $a_i(v)=a_{B_i}(v)$.
We have $a_i(v)\ge2$ and $\sum_i a_i(v)=d_s(v)$ which implies $a_i(v)a_j(v)\ge 2(a_i(v)+a_j(v))-4$,
\[
        \sum_{1\le i<j\le d_b(v)}a_i(v)a_j(v)\ge 2(d_b(v)-1)  \sum_{1\le i\le d_b(v)}a_i(v)-4\binom{d_b(v)}{2} = 2(d_b(v)-1)(d_s(v)-d_b(v)).
\]
Therefore the contribution of $v$ to $P+Q$ as the middle vertex is at least
\begin{align*}    
        \binom{d_s(v)}{2}+2(d_b(v)-1)(d_s(v)-d_b(v))=\left(\frac12d_s(v)^2+2d_b(v)d_s(v)-2d_b(v)^2\right)-\frac52d_s(v)+2d_b(v)\\
        \geq \left(\frac{d_s(v)}{\sqrt2}+(2-\sqrt2)d_b(v)\right)^2-\frac52d_s(v)+2d_b(v).
\end{align*}
Since $d_s(v)\ge2d_b(v)$. Summing over all vertices gives
\[
        P+Q\ge \sum_v\left(\frac{d_s(v)}{\sqrt2}+(2-\sqrt2)d_b(v)\right)^2-O(m) \ge \frac1n\left(\frac1{\sqrt2}\sum_v d_s(v)+(2-\sqrt2)\sum_v d_b(v)\right)^2-O(m),
\]
by the Cauchy-Schwarz inequality.
By \eqref{eq:sums_low}, together with \eqref{eq:upperPQ} we get 
\[
        n^2+O(m)\geq P+Q+O(m)\geq  \frac1n\left(\frac1{\sqrt2}\sum_v d_s(v)+(2-\sqrt2)\sum_v d_b(v)\right)^2 \geq (2+\sqrt2)^2\frac{m^2}{n},
\]
 proving Theorem \ref{thm:main}.


\section*{Acknowledgments}

The research of Salia was supported by the National Science Centre grant 2021/42/E/ST1/00193.

During the preparation of this manuscript, the authors used ChatGPT
(OpenAI) to assist in exploring proof formulations, checking intermediate
arguments, and improving exposition. The authors independently verified
all mathematical claims, proofs, and references, revised the text as
needed, and take full responsibility for the content of the manuscript.

\bibliographystyle{abbrv}
\bibliography{references.bib}

\end{document}